\documentclass[reqno]{amsart}

\usepackage{enumitem}
\usepackage{hyperref}
\usepackage{amsthm, amssymb, amsmath,mathrsfs}
\usepackage[utf8]{inputenc}

\newtheorem{Th}{Theorem}[section]
\newtheorem{Cor}{Corollary}[section]
\newtheorem{Lemma}{Lemma}[section]
\newtheorem{De}{Definition}[section]
\newtheorem*{Rem}{Remark}

\newcommand{\D}{{\mathbb D}}

\newcommand{\R}{{\mathbb R}}
\newcommand{\Z}{{\mathbb Z}}

\newcommand{\C}{{\mathbb C}}
\newcommand{\T}{{\mathbb T}}

\newcommand{\bD}{{\mathbb D}}

\newcommand{\cA}{\mathcal A}
\newcommand{\cC}{\mathcal C}
\newcommand{\cE}{\mathcal E}

\newcommand{\cJ}{{\mathcal J}}
 \newcommand{\cL}{\mathcal L} 
\newcommand{\cM}{\mathcal M} 
\newcommand{\cO}{\mathcal O}
\newcommand{\cR}{\mathcal R} 
\newcommand{\cS}{\mathcal S}
\newcommand{\cX}{\mathcal X}

\newcommand{\ga}{{\mathfrak a}}
\newcommand{\gb}{{\mathfrak b}}

\newcommand{\gm}{{\mathfrak m}}

\newcommand{\la}{\lambda}
\newcommand{\si}{\sigma}
\newcommand{\om}{\omega}
\newcommand{\te}{\theta}
\newcommand{\al}{\alpha}
\newcommand{\be}{\beta}
\newcommand{\di}{{\mathbf{\rm{diag}}}}
 \newcommand{\Res}{{\mathbf{\rm{Res}}}}

\newcommand{\beq}{\begin{equation}}
\newcommand{\eeq}{\end{equation}}

\newcommand{\ov}{\overline} 
 
 \newcommand{\beqq}{\begin{equation*}}
\newcommand{\eeqq}{\end{equation*}}

\newcommand{\lng}{\langle}
\newcommand{\rng}{\rangle}

\newcommand{\mbf}{\mathbf}
\newcommand{\mbfx}{\mathbf {x}} 

\newcommand{\supp}{\mbox{supp} \ }

\newcommand{\I}{\mathrm{i}}

\newcommand*{\mailto}[1]{\href{mailto:#1}{\nolinkurl{#1}}}

\newcommand{\msc}[1]{\href{http://www.ams.org/msc/msc2010.html?t=&s=#1}{#1}}

\numberwithin{equation}{section}

\begin{document}
\title{Discrete Multichannel Scattering  with step-like potential}
\author[I. Alvarez-Romero]{Isaac Alvarez-Romero}
\address{Department of Mathematical Sciences,
Norwegian University of
Science and Technology, NO--7491 Trondheim, Norway}
\email{\mailto{isaac.romero@math.ntnu.no}\\
\mailto{isaacalrom@gmail.com}}

\author[Yu. I. Lyubarskii]{Yurii  Lyubarskii}
\address{Department of Mathematical Sciences,
Norwegian University of
Science and Technology, NO--7491 Trondheim, Norway}
\email{\mailto{yura@math.ntnu.no}}

\thanks{{\it This research has been supported by the Norwegian Research Council project DIMMA 213638.}}

\begin{abstract}
We study direct and inverse scattering problem for systems of interacting particles, having  web-like structure. Such systems  consist of a finite number of semi-infinite chains attached to the central part formed by a finite number of particles. We assume that  the semi-infinite channels are homogeneous at infinity, but the limit values of the coefficients may vary from one chain to another.
\end{abstract}

\maketitle

\medskip

\section{Introduction}
The aim of this article is to study the direct and  inverse   problems  for small oscillations near equilibrium position for a system of particles  $\cA=\{\alpha, \beta, \ldots \}$ which
interact with each other and perhaps with an external field. The interaction is described  by the matrix 
$$
\cL= \left (  L(\alpha, \beta)  \right )_{\alpha,\beta\in \cA}.
$$
We say that the particles $\alpha$ and $\beta$ interact with each other if $L(\alpha, \beta) \neq 0$ 
and we assume that each particle interacts with at most a finite number of its neighbours:  $\# \{ \beta, L(\alpha, \beta)\neq 0 \} < \infty$ for each $\alpha \in \cA$.\\
Our the system has a "web-like" structure: it includes  a finite set  of  "channels", i.e. semi-infinite chains of particles attached to a "central part"  formed by a finite number of interacting particles. \\ \\
Given a set of particles $\cX$ we denote by $\cM(\cX)$ and $l^2(\cX)$ the spaces of all functions on $\cX$ and square summable functions on $\cX$ respectively.\\ \\
The matrix $\cL$ is related to the  Hessian matrix of the potential energy near    equilibrium position, so we always assume that all $L(\alpha,\beta)$ are real and
\beqq
\cL > 0,
\eeqq
here $\cL$ is considered as an operator in $l^2(\cA)$.  In what follows we do not distinguish between matrices and the corresponding linear operators.\\ \\ 
After separation of variables one arrives to the spectral problem 
\beq
\label{eq:01}
\cL\xi = \lambda \xi; \ \xi=\{\xi(\alpha)\}
\in l^2(\cA).
\eeq
which can be considered as a discrete version of spectral problems for quantum graphs, see  \cite{C,GP,GS,PF} as well as later articles \cite{K,KN,KS}.\\ \\
The web-like structure of the system allows us to treat the spectral problem \eqref{eq:01} as a scattering problem: the points of continuos spectra correspond to frequences of incoming and outcoming waves which are propagating along the channels; the points of discrete spectra correspond to proper oscillations of the system. The spectral data includes transmission and reflection coefficients: we consider waves incoming along one of the channels and observe (at infinite ends of the channels) how do such waves come through the system. The direct problem is to determine the spectral data through the characteristics of the system. The inverse problem deals with recovering of characteristics of the system from  the spectral data. This is a classical setting of the scattering problem, see e.g. \cite{N}.\\ \\
This work is a continuation of  \cite{LM} which considers the case of the same wave propagation speed along all channels. Now we assume that each channel has its own speed. The problem then becomes  more complicated: different points of continuous  spectra may have different multiplicities; we do not have  a single scattering matrix for the whole spectra; besides  the generalized eigenfunctions may exponentially decay along some of the channels. These eigenfunctions  need to be treated specially, it does not make sense to observe the phase of a wave which decays exponentially at inifinity, just their amplitudes can be taken into account. Respectively, for such waves, only absolute values of the transmission coefficients can be included to spectral data.\\ \\
As in \cite{LM} the problem can be reduced to a system of difference Schr{\"o}dinger equations  on semi-infinite discrete string with the initial conditions related to the way the channels   are attached to the central part. In our setting the potentials have different limits for different equations in this system. In the classical case of infinite string this corresponds to a step-like potential, such problem for continuous case has been studied for example in \cite{BF, CK, GNP},   for the Jacobi operators the case of step-like quasi-periodic potential has been considered in      \cite{EMT}.  We modify techniques of these articles, especially those in \cite{CK}, in order to make it applicable for the graph setting. \\ \\
The scattering data are determined as the set of scattering coefficients,  eigenvalues, and also as normalization   constants of the corresponding eigenfunctions. These scattering data are associated with the data which can be measured in an experiment at infinite ends of channels. \\ \\
We solve the direct problem, i.e, description of scattering data from physical characteristic of the system, and (under additional assumptions) the inverse problem, i.e.  reconstruction the   characteristics of the channels from given spectral data, that is, we reconstruct the matrix $\cL$ on the channels. In order to make this reconstruction we reduce the problem to a discrete analog of Marchenko equation, which is known to have unique solution and actually admits numeric implementation.\\ \\
In this article we do not discuss possibility of reconstruction of  data which is related to the central part of the graph. In order to have such reconstruction, one needs to demand additional sparsity  conditions of the central part.  We refer  the reader to the recent book \cite{MS}, which contains examples of such conditions.\\ \\
The article is organised as follows: in the next section we describe the system, derive the \emph{boundary condition} which allows us to treat the problem as a system of equations on a string. In section \ref{sec3} we consider the characteristics of the channels, introduce the Jost solutions and also describe the spectral data related to the continuos spectra. In section \ref{sec4} we collect some known results, see e.g. \cite{LM1,VA,Te},  as well as some new ones about solutions of finite difference equations. These results will be used in the sequel. In particular in section \ref{sec:Construction} we construct the special solutions, i.e. solutions which correspond to wave incoming along one of the channels. Using these solutions we study the structure of discrete spectra, this is done in section \ref{SingU(k)}.\\
As it was already mentioned, no scattering matrix can be defined for the whole spectra, however the scattering coefficients possess some symmetry which plays the crucial role in our construction. This symmetry is described in section \ref{sec5}. In section \ref{sec8} we return to discrete spectra and connect each eigenvalue of the problem with normalized matrix of eigenfucntions which gives the energies, completing the set of spectral data.\\
In section \ref{sec9} we collect all previous results and finally obtain equations of the inverse scattering problem. Section \ref{sec10} contains some concluding remarks.

\medskip

\noindent 
{\bf Acknowledgments.}
We  thank V. A. Marchenko, who suggested the problem, and also for numerous fruitful discussions.

\section{Geometry  of the system and the boundary condition}\label{sec2}
We consider the systems $\cA$ which have a  "web-like" structure.   
Namely $\cA=\cA_0\cup \cA_1$ where $\cA_1$ is a  central part,  and $\cA_0$ is a union of a finite number of  semi-infinite channels.   
For such web-like  systems the   inverse spectral problem can be treated as an  inverse scattering problems: sending a wave along one of the channels 
and observing how does it pass through the system we are trying to reconstruct the characteristics of the system, i.e. the values $L(\alpha,\beta)$.

\medskip

\noindent{\bf Definition} A sequence of particles $\sigma=\{\alpha(p)\}_{p=0}^\infty$ is a channel if, for $p>0$ the particle  $\alpha(p)$ interacts with the particles  $\alpha(p-1)$ and $\alpha(p+1)$ only (and, perhaps with  the external field), while $\alpha(0)$ interacts with 
$\alpha(1)$ and some other particles in $\cA$ which do not belong to $\sigma$. 

\medskip

We will use the following notation:

\noindent - the set of all channels is $\cC$, the channels will be denoted   by $\sigma$, $\nu$, $\gamma$ etc.; 

\noindent - the particles in $\sigma \in \cC$ are $\sigma(0)$, $\sigma(1)$, $\sigma(2)$, $\ldots$,
 the point $\sigma(0)$ is called  the attachment point of $\sigma$;   

\noindent - $\Gamma:= \{\sigma(0)\}_{\sigma \in \cC}$, \ $\cA_0:=\cup_{\sigma\in \cC} \cup_{k=1}^\infty {\sigma(k)}$,  \ 
$\cA_1:= \cA \setminus \cA_0$.
 
\noindent - For $\sigma\in \cC$, $k=1,2,\ldots$ we also denote 
\begin{equation}\label{sec1:1.2}
-\gb_\sigma(k-1)=L(\sigma(k-1), \sigma(k)),\qquad  \ga_\sigma(k)=L(\sigma(k), \sigma(k)),
\end{equation}
so equation 
  \eqref{eq:01}  on the channel $\sigma$ takes the form:
  \beq
  \label{eq:02}
   -\gb_\sigma(k-1)\xi(\sigma(k-1)) +\ga_\sigma(k)\xi(\sigma(k))-\gb_\sigma(k)\xi(\sigma(k+1)) = \la \xi(\sigma(k)).
 \eeq

\noindent We assume that the number of particles in the central part as well as the number of channels is finite:
\begin{equation*}
M:=\sharp \mathcal{A}_1<\infty,\qquad \sharp \mathcal{C}<\infty
\end{equation*}
Also (for simplicity) we assume $\sigma(0)\neq\nu(0)$, $\sigma,\nu\in\mathcal{C}$, $\sigma\neq\nu$

\subsection{Boundary conditions}  
 
 Let $\xi \in \cM(\cA)$ be a solution to \eqref{eq:01}. Then it meets \eqref{eq:02}  and also, for each $\alpha \in \cA_1$, 
 \beqq
 \la \xi(\alpha) - \sum_{\beta\in \cA_1}L(\alpha,\beta) \xi(\beta) = \sum_{\beta\in \cA_0}L(\alpha,\beta) \xi(\beta)
 \eeqq  
 The only pairs $(\alpha,\beta)\in \cA_1\times \cA_0$ for which $L(\alpha,\beta)\neq 0$ are of the form $(\sigma(0),\sigma(1))$, $\sigma\in \cC$,
 so with account of \eqref{sec1:1.2} this relation can be written as 
 \beq
 \label{eq:04}
 \la \xi(\alpha) - \sum_{\beta\in \cA_1}L(\alpha,\beta) \xi(\beta) =
           \begin{cases}  
                       -\gb_\nu(0)\xi(\nu(1)), & \alpha=\nu(0)\in \Gamma, \\
                       0, & \alpha  \in \cA_1\setminus  \Gamma
           \end{cases}
 \eeq 
Consider the matrix
\beqq
\label{eq:05}
\cL_1= \left ( L(\alpha,\beta)\right )_{\alpha,\beta\in \cA_1}.
\eeqq
Being a truncation of $\cL$, the matrix $\cL_1$ also is strictly positive. Let $0<\la_1\leq \la_2 \leq \ \ldots \leq \la_M$ and $p_1, \ \ldots , p_M$ be 
its eigenvalues and the corresponding normalised eigenvectors.  We may choose $p_j$'s to be real-valued. For $\la\not\in \{\la_j\}_{j=1}^M$ the operator $\cL_1-\la I$ is invertible and 
\beqq
\label{eq:06}
\left ( \cL_1-\la I\right )^{-1}= R(\lambda)=\left ( r(\alpha,\beta;\la)  \right )_{\alpha,\beta \in \cA_1}, \quad 
            r(\alpha,\beta;\la) =\sum_{j=1}^M \frac {p_j(\alpha)p_j(\beta)}{\la_j-\la}     .        
\eeqq
Relation \eqref{eq:04} can be than written as
\beq
\label{eq:07}
\xi(\alpha)= \sum_{\nu\in \cC} r(\alpha, \nu(0); \la) \gb_\nu(0)\xi(\nu(1)); \ \alpha \in \cA_1.
\eeq
In particular for $\alpha=\sigma(0)$, $\sigma \in \cC$ we obtain
\beq
\label{eq:08} 
\xi(\sigma(0))= \sum_{\nu\in \cC} r(\sigma(0), \nu(0); \la) \gb_\nu(0)\xi(\nu(1)). 
\eeq
After introducing vector notations 
\beqq
\label{eq:08a}
\vec{\xi}(k):= (\xi_\si(k))_{\si \in \cC}, \ k=0,1, \ \ldots, \  \cR(\lambda)=(r(\sigma(0), \nu(0); \la))_{\si,\nu\in\cC}
\eeqq
we can rewrite \eqref{eq:08} as 
\beq
\label{eq:08b}
\vec{\xi}(0)=\cR(\la)\di\{\gb_\nu(0)\}_{\nu\in \cC}  \ \vec{\xi}(1).
\eeq
This relation connects the values of solution $\xi$ on $\cA_0$ and on $\Gamma$.  It plays the role of the boundary condition for vector-valued scattering problem.\\
The following statement holds.

\begin{Th}( See theorem 1 in \cite{LM})
Let $\la\not \in \{\la_j\}_{j=1}^M$. A function $\xi\in \cM(\cA_0\cup \Gamma)$ admits prolongation to a function $\xi \in \cM(\cA)$ satisfying 
\eqref{eq:01} if and only if it satisfies \eqref{eq:02}  and also the boundary condition \eqref{eq:08}. The prolongation is unique and can be defined by 
\eqref{eq:07}.
\end{Th}

\section{Characteristics of the channels and spectral data}\label{sec3}

\subsection{Jost solutions }

We assume  that the channels are asymptotically homogeneous at infinity. Namely,
for each $\sigma \in \cC$ there exist $b_\sigma$ and $a_\sigma$ such that 
$$
\ga_\sigma(k) \to a_\sigma, \  \gb_\sigma(k) \to b_\sigma  \ \text {as} \ k\to \infty.
$$
Moreover
\beq
\label{eq:09}
\sum_{k=1}^\infty k \{ |\ga_\sigma(k)-a_\sigma| + |\gb_\sigma(k)-b_\sigma|\} < \infty.
\eeq 

This relation provides existence of Jost solutions on each channel $\sigma \in \cC$. It is well known, see e.g. \cite{MS,Te}, that  for each $\sigma\in \cC$ 
there is a family $\{e_\sigma(k, \theta)\}_{k=0}^\infty$ of functions holomorphic inside the open disk $\D$, continuous 
 up to the boundary and, for each $\theta\in \T$ and $k\geq 1$,
 \beq
  \label{eq:10}
   -\gb_\sigma(k-1)e_\sigma(k-1,\theta) +\ga_\sigma(k)e_\sigma(k,\theta)-\gb_\sigma(k)e_\sigma(k+1,\theta) = \la_\sigma(\theta)e_\sigma(k,\theta),   
 \eeq  
where 
\beq
\label{eq:11}
 \la_\sigma(\theta)=a_\sigma-b_\sigma \left ( \theta+\theta^{-1}\right ).
\eeq
 In addition
 \beqq
 \label{eq:12}
e_\sigma(k, \theta)= \theta^k (1+o(1)) \ \text{as} \ k \to \infty
\eeqq
uniformly with respect to $\theta \in\overline{\D}$ 

The functions $e_\sigma(k, \theta)$ admit the representation
\beq
\label{eq:13}
e_\sigma(k,\theta)= c_\sigma(k)\sum_{m\geq k} a_\si(k,m)  \theta^m, \ k=0,1,\ldots .
\eeq
Here 
\beqq
\label{eq:14}
 c_\sigma(k)=\prod_{p=k}^\infty \frac {b_\sigma}{\gb_\sigma(p)}, \ a_\sigma(k,k)=1 \  \text{and}  \ \lim_{k\to \infty} \sum_{m\geq k+1} |a_\sigma(k,m)| =0
 \eeqq
For $k\geq 1$ the coefficients $\ga_\sigma(k)$, $\gb_\sigma(k)$ can be expressed through the coefficients of the functions $e_\sigma(k, \theta)$ 
\begin{gather*}
\label{eq:15}
\frac{\ga_\sigma(k)-a_\sigma}{b_\sigma}= a_\sigma(k-1,k)-a_\sigma(k,k+1); \\
\label{eq:16}
\frac{\gb_\sigma^2(k)}{b_\sigma^2} =  \frac{\ga_\sigma(k)-a_\sigma}{b_\sigma} a_\sigma(k,k+1)+a_\sigma(k,k+2)-a_\sigma(k-1,k+1)+ 1.
\end{gather*}
These relations can be obtained by substituting representation \eqref{eq:13} in \eqref{eq:10} and then comparing the coefficients 
with the same powers of $\theta$.\\ \\
In this article we assume that the coefficients $\ga_\sigma(k),\gb_\sigma(k)$ approach their limit values faster than \eqref{eq:09}. Namely we assume that, for some $\epsilon>0$,
\begin{equation*}
\sum_{k=1}^\infty (1+\epsilon)^k\{ |\ga_\sigma(k)-a_\sigma| + |\gb_\sigma(k)-b_\sigma|\} < \infty.
\end{equation*}
Under this condition the functions $e_\sigma(k,\theta)$ admit holomorphic prolongation to some vicinity of the unit disk $\mathbb{D}$, see e.g. chapter 10 in  \cite{Te}. This allows us to avoid extra technicalities which are necessary in the general case.
\subsection{Spectra, scattering data corresponding to the continuous spectra}

As the parameter $\theta$  runs through  the unit circle $\T$, the corresponding value $\la_\sigma(\theta)$ defined by  \eqref{eq:11} runs through the segment
\beqq
\label{eq:17}
I_\sigma= [a_\sigma-2b_\sigma, a_\sigma+2b_\sigma].
\eeqq
We assume for the simplicity that $\cup_{\sigma\in \cC}(a_\sigma-2b_\sigma, a_\sigma+2b_\sigma) $ is a connected set and choose $a,b\in \R$ so that 
\beqq
\label{eq:18}
I=\cup_{\sigma\in \cC} I_\sigma = [a-2b,a+2b]
\eeqq
We need the  Zhukovskii mappings  $W,\ W_\sigma : \D \to \C$:
\begin{gather}
\label{eq:19}
W : \omega \mapsto   a-b(\omega+\omega^{-1}),     \  W_\sigma : \theta\mapsto   a_\sigma-b_\sigma(\theta+\theta^{-1}), \\
\label{eq:20}
 \theta_\sigma(\omega)= W_\sigma^{-1} \circ W (\omega).
\end{gather}
Further we choose the set $J_\sigma \subset  [-1,1]$,   so that $\theta_\sigma$ maps $\D$ onto 
 $\D_\sigma:=\D\setminus J_\sigma $ and let 
 \beqq
 \label{eq:21}
 T_\sigma^+= \theta_\sigma^{-1} (\T), \   T_\sigma^-= \theta_\sigma^{-1}(J_\sigma)= \T\setminus  T_\sigma^+.
\eeqq
The set $J_\si$ consists of one or two segments attached to the points $\pm 1$. \\
The  functions $e(k,\theta_\sigma(\omega))$ are holomorphic in $\D$. Moreover, for $\omega\in T_\sigma^+$, the functions  $e(k,\theta_\sigma(\omega)^{-1})$
are well defined.\\ \\
By a special solution which correspond to $\sigma\in \cC$ we mean the function $\psi^\sigma(\alpha,\omega)$,  $\omega\in \D$, $\alpha\in \cA$
  with the following properties:\\ \\
\medskip
1. For each $\alpha\in \cA_0$ the function $\psi^\sigma(\alpha, \cdot)$ is holomorphic in a vicinity of  $\overline{\D}$,  except, perhaps  a finite set $\cO\subset \overline{\D}$ where it has poles.\\ \\
\medskip
2. For each $\omega\in \overline{\D}\setminus \cO$,  the function    $\psi^\sigma(\cdot, \omega)\in \cM(\cA)$ meets the equation 
 \beqq
 \label{eq:22}
 \cL \psi^\sigma ( \cdot,\omega)  = \la\psi^\sigma(\cdot, \omega),  \   \la=a+b(\omega+\omega^{-1})
 \eeqq
\medskip
3.  For $\gamma\neq \sigma$ there exist functions $s_{\sigma\gamma}(\omega)$, meromorhpic in a vicinity of  $\overline{\D}$  and a function $s_{\sigma\sigma}(\omega)$   continuous on $T^+_\sigma$ such that  
 the function $\psi^\sigma$ has the following representation along the channels:
 \beq
 \label{eq:23}
 \psi^\sigma (\gamma(k), \omega)=\begin{cases}
                                                               e_\sigma(k, \theta_\sigma(\omega)^{-1}) + s_{\sigma \sigma}(\omega)e_\sigma(k, \theta_\sigma(\omega)),  & \gamma=\sigma; \\
                                                                 s_{\gamma\sigma}(\omega)e_\gamma(k, \theta_\gamma(\omega)), & \gamma\neq \sigma
                                                        \end{cases}, \
                                                        \omega \in T^+_\sigma,
 \eeq
 and  for $\omega \in \D$
 \beq
 \label{eq:23a} 
 \psi^\sigma (\gamma(k), \omega)=\begin{cases}
    \theta_\sigma(\omega)^{-k} (1+o(1)), \ \mbox{as} \ k\to \infty, & \gamma=\sigma \\
            s_{ \gamma\sigma}(\omega)e_\gamma(k,\theta_\gamma(\omega)), &  \gamma\neq \sigma.
                                              \end{cases}, \ \omega \in \overline{\D}\setminus T^+_\sigma.
 \eeq
The {\em scattering coefficients}, $s_{\sigma\gamma}(\omega)$ are  the elements of the scattering matrix, corresponding
    to the problem \eqref{eq:01}. We mention  that so far they are defined on different arcs $T_\sigma^+\subset \T$. In 
      section \ref{sec:Construction} we will prove existence and uniqueness of such  special solutions for each $\sigma \in \cC$.\\
The solution $\psi^\si$ corresponds to the wave incoming along the channel $\si$ and distributing through the whole system.
This wave is well-defined  if $\omega\in T^+_\sigma$ (respectively $\lambda\in I_\sigma$), for these values the equation \eqref{eq:02} along the channel $\sigma$ admits two independent bounded solutions corresponding to  in- and out-coming  waves.  For 
$\omega \in T^+_\sigma\cap T^+_\gamma$ the wave incoming along $\sigma$ generates an outcoming wave along the 
channel $\gamma$, this wave can be observed at infinity. For $\omega \in T^+_\sigma\cap T^-_\gamma$ the wave
 incoming along $\sigma$ generates an exponentially decaying wave  $\gamma$,  observation of phase of such wave at
  infinity  is virtually impossible, one can measure the absolute values only. These reasoning  explain the definition  
  of the spectral data.

 \begin{De}
By continuous spectral data corresponding to the channel $\sigma \in \cC$ we mean the set of functions
\beqq
\cS_\sigma= \{ s_{\gamma \sigma}(\omega):\text{ }\omega\in T_\sigma^+\cap T_\gamma^+;\text{ }\gamma\in \cC\}\cup\{ |s_{ \gamma\sigma}(\omega)|:\text{ }\omega\in T_\sigma^+\cap T_\gamma^-; \text{ } \gamma\in \cC \}
\eeqq
By full continuous spectral data we mean
\beqq
\cS= \cup_{\sigma \in \cC}  \cS_\sigma.
\eeqq
 \end{De}
Later in sections \ref{SingU(k)} and \ref{sec8} we will discuss discrete spectrum, which corresponds to eigenfunctions of  the operator $\cL:l^2(\cA)\to l^2(\cA)$.

\section{Properties of solutions of finite difference equations}\label{sec4}

We need to establish some properties of solutions to \eqref{eq:01} as well as solutions to 
difference equations along the channels. In this section we collect   known   (see e.g. \cite{LM1,VA,Te} ) properties of solutions of
the finite-difference equation as well as some new statements related to solutions of \eqref{eq:01}
\beq
\label{eq:30}
-\gb(k-1)x(k-1)+\ga(k)x(k)-\gb(k)x(k+1)=\lambda x(k), \quad k=1, 2, \ldots \,.
\eeq
with real coefficients $\ga(k)$, $\gb(k)$. We also assume  that $\gb(k)>0$ in order that the Jost solutions will be well defined, see e.g. \cite{Te}.\\ \\ 
Given two functions $x=x(k)$, $y=y(k)$ on the  integers we define their
\textit{Wronskian}   $\{x,y\}$ as
\beqq
\label{eq:31}
\{x,y\}(k)= x(k){y(k+1)}- x(k+1){y(k)}, \quad k=0,1,\ldots \,.
\eeqq

\medskip

{\bf 1.} Let $x$, $y$  be  solutions to (\ref{eq:30}). Then, for all $N$,
\beq
\label{eq:32}
\gb(N)\{x,y\}(N)-\gb(0)\{x,y\}(0)=0
 \eeq
and
\beq
\label{eq:33}
\gb(N)\{x,\ov{x}\}(N)-\gb(0)\{x,\ov{x}\}(0)=
       (\lambda - \bar{\lambda})\sum_{k=1}^N |x(k)|^2.
\eeq
If in addition $\lambda=\lambda(\om)$ and $x=x(k,\om)$ are differentiable functions of a parameter $\om$, then
\beq
\label{eq:34}
\gb(N)\{\dot{x},\ov{x}\}(N)-\gb(0)\{\dot{x},\ov{x}\}(0)=
      \dot{\lambda}\sum_{k=1}^N|x(k)|^2 +
            (\lambda-\bar{\lambda})\sum_{k=1}^N\dot{x}(k)\overline{x(k)} \, ,
\eeq
here  and in what follows the dot denotes  derivative with respect to $\om$.\\  \\
If in addition 
\beqq
 \label{eq:35}
 \sum_{k=1}^\infty k( |\gb(k)-b|+|\ga(k)-a|) < \infty,
 \eeqq
and $e(\theta)=e(k,\theta)$ are the corresponding Jost solutions, relation \eqref{eq:33}  yields
\beq
\label{eq:36}
\gb(0)\{e(\theta), \ov{e}(\theta) \}(0) =\begin{cases} 
                                                                   b(\ov{\theta}-\theta), & |\theta|=1; \\
                                                                  b (\ov{\te}-\te) (|\te|^{-2}-1)\sum_{k=1}^\infty|e(k,\te)|^2, &  |\te|<1.                                                                   
                                                           \end{cases}        
\eeq

\begin{Lemma}
\label{le:02a}
Let $\om \in \T\cup(-1,1)$ and $\si\in \cC$ be such that $\te_\si(\om)\in (-1,1)$. Then 
\beq 
\label{eq:38a07}
-\gb_\si(0)\{\dot e_\si,(\te_\si(\om)) e_\si(\te_\si(\om))\}_\si(0)= \dot \la\sum_{k=1}^\infty e_\si(k,\te_\si(\om))^2.
\eeq
\end{Lemma}
This is an immediate consequence of \eqref{eq:34}.\\ 

\medskip

{\bf 2.}  We will consider Wronskians, which correspond to various channels $\sigma\in\mathcal{C}$. Given two functions $\xi,\eta\in \cM(\cA)$  and $\si\in \cC$ we denote
\beqq
\label{eq:37}
\{\xi,\eta\}_\si (k):=
                                 \{\xi(\si(k))\eta(\si(k+1))-\xi(\si(k+1))\eta(\si(k))\}
\eeqq

\smallskip

\begin{Rem}\label{Rem:01}
If both $\xi$ and $\eta$ are solutions of \eqref{eq:01} it follows from \eqref{eq:32} that the quantity 
$\gb_\si(k)\{\xi,\eta\}_\si (k)$ depends on $\si$ only. 
\end{Rem}

\smallskip

\begin{Lemma}
\label{le:01}
Let $\xi,\eta\in \cM(\cA)$ be solutions to \eqref{eq:01}. Then 
\beq
\label{eq:38}
\sum_{\si\in \cC} \gb_\si(0)\{\xi,\eta\}_\si (0) = 0
 \eeq
\end{Lemma}
        \begin{proof}
For $\alpha\in \cA_1$ we have 
\begin{gather*}
\sum_{\beta\in \cA_1}L(\alpha,\beta)\xi(\beta) \ + \  \sum_{\beta\in \cA_0}L(\alpha,\beta)\xi(\beta) = \la\xi(\alpha), \\
\sum_{\beta\in \cA_1}L(\alpha,\beta)\eta(\beta) \ + \  \sum_{\beta\in \cA_0}L(\alpha,\beta)\eta(\beta) = \la\eta(\alpha),
\end{gather*}
We multiply these relations by $\eta(\alpha)$ and $\xi(\alpha)$ respectively. Summation with respect to $\alpha\in\cA_1$
gives
\begin{gather*}
\sum_{\alpha, \beta\in \cA_1}L(\alpha,\beta)\xi(\beta)\eta(\alpha) \ + \  \sum_{\alpha\in \cA_1, \beta\in \cA_0}L(\alpha,\beta)\xi(\beta)\eta(\alpha) 
= \la \sum_{\alpha\in\cA_1}\xi(\alpha)\eta(\alpha), \\
\sum_{\alpha, \beta\in \cA_1}L(\alpha,\beta)\eta(\beta)\xi(\alpha) \ + \  \sum_{\alpha\in \cA_1,\beta\in \cA_0}L(\alpha,\beta)\xi(\alpha)\eta(\beta) 
 = \la \sum_{\alpha\in\cA_1}\xi(\alpha)\eta(\alpha).
\end{gather*}
The right-hand sides in these relations coincide as well as the first terms in the left-hand sides (since $L(\alpha,\beta)=L(\beta,\alpha)$). Therefore
\beqq
 \sum_{\alpha\in \cA_1, \beta\in \cA_0}L(\alpha,\beta) \left ( \xi(\beta)\eta(\alpha) - \xi(\alpha)\eta(\beta) \right )=0.
\eeqq
This proves the lemma since the only option for $L(\alpha,\beta)\neq 0$ for $\alpha\in \cA_1$, $\beta\in \cA_0$ is $\alpha=\si(0)$, 
$\beta=\si(1)$ for some $\si\in \cC$ and in this case $L(\si(1),\si(0))=-\gb_{\si} (0)$.
       \end{proof}
       
   \medskip 
   
   {\bf 3.}  We need a special statement in order to calculate the energy of eigenfunctions of the operator $\cL$. 
   
   \begin{Lemma}
   \label{le:02}
   Let a function $\xi(\alpha)=\xi(\alpha,\om)\in \cM(\cA)$ be differentiable with
    respect to $\om$ in a neighborhood of $\hat{\om}\in \overline{\D}$, $\Im \la(\hat{\om})=0$ and $\xi(\al)$ satisfy the equation
    \beq
    \label{eq:38a01}  
     \la(\om)\xi(\al,\om)= \sum_{\be\in \cA}L(\al,\be)\xi(\be,\om).
     \eeq   
     Let also a function $\eta(\al)$ meet this equation for $\om =\hat{\om}$. 
     Then 
 \beqq
 \label{eq:38a02}
 \dot{\la}(\hat\om)\sum_{\al\in \cA_1}\xi(\al,\hat\om)\ov{\eta(\al)}=
 \sum_{\si\in \cC}\gb_\si(0)\{\dot{\xi}(\hat \om), \ov{\eta}\}_\si(0).
 \eeqq
 In particular 
 \beqq
 \label{eq:38a03}
 \dot{\la}(\hat \om)\sum_{\al\in \cA_1}|\xi(\al,(\hat\om))|^2=
 \sum_{\si\in \cC}\gb_\si(0)\{\dot{\xi}(\hat \om), \ov{\xi(\hat \om)}\}_\si(0).
 \eeqq
 \end{Lemma}

\begin{proof}
The statement uses the same idea as in lemma 2.1 in \cite{LM}, we repeat here the construction:\\
Differentiate \eqref{eq:38a01} with respect to $\om$:
\beqq
    \label{eq:38a03}  
     \dot{\la}(\om)\xi(\al,\om)+ {\la}(\om)\dot{\xi}(\al,\om)= \sum_{\be\in \cA}L(\al,\be)\dot{\xi}(\be,\om).
     \eeqq   
Besides 
\beqq
 \label{eq:38a04}
 \la(\hat{\om})\ov{\eta(\al)}= \ov{ \la(\hat{\om})\eta(\al)}=\sum_{\beta\in \cA}
    L(\al,\be) \ov{\eta(\be)}.
 \eeqq   
Combining these equations we obtain
 \beqq
 \label{eq:38a05}
  \dot{\la}(\hat{\om}) {\xi}(\al,\om)\ov{\eta(\al)} = \sum_{\beta\in \cA}
    L(\al,\be) [ \dot{\xi}(\beta, \hat{\om})  \ov{\eta(\al) } - \dot{\xi}(\al, \hat{\om})  \ov{\eta(\be)}],
\eeqq
and
 \begin{multline*}
 \label{eq:38a06}
  \dot{\la}(\hat{\om})\sum_{\alpha\in\mathcal{A}_1} {\xi}(\al,\om)\ov{\eta(\al)}= \sum_{\al\in\cA_1}\sum_{\be\in \cA_1}
  L(\al,\be) [ \dot{\xi}(\beta, \hat{\om})  \ov{\eta(\al) } - \dot{\xi}(\al, \hat{\om})  \ov{\eta(\be)}] + \\
  \sum_{\al\in\cA_1}\sum_{\be\in \cA_0}
  L(\al,\be) [ \dot{\xi}(\beta, \hat{\om})  \ov{\eta(\al) } - \dot{\xi}(\al, \hat{\om})  \ov{\eta(\be)}].
  \end{multline*}
The first summand in the right-hand side of this relation vanishes because it is anti-symmetric in $\al$ and $\be$.
  In the second summand the only non-zero coefficient $L(\al,\be)$  appears in the case $\al=\si(0)$, $\be=\si(1)$ for some $\si\in \cC$ and $L(\si(0),\si(1))=-\gb_\si(0)$.
       \end{proof}

%
%
%
\section{Construction of the special solutions}\label{sec:Construction}

Consider the diagonal matrices 
\beqq
\label{eq:50}
B:=\di\{b_\si\}_{\si\in \cC}, \ B(0):=\di\{\gb_\si(0)\}_{\si\in \cC}, 
\eeqq
and also the matrix-functions in $\bar{\D}$
\beqq
\label{eq:51}
\cE(k,\om):=\di \{e(k,\te_\si(\om))\}_{\si\in \cC}, \   P(k,\om):= \di \{p_\si(k,\om)\}_{\si\in \cC},
\eeqq
here $\{p_\si(k,\om)\}_{k=0}^\infty$  satisfy the equation on the channels:
\begin{equation*}
-\gb_\sigma(k-1)p_\si(k-1,\om)+\ga_\sigma(k)p_\si(k,\om)-\gb_\sigma(k)p_\si(k+1,\om)=\lambda(\omega)p_\si(k,\om)
\end{equation*}
for $k=1,2,\dots$, $\sigma\in\mathcal{C}$ and  with  boundary conditions  
\beq
\label{eq:52}
 p_\si(0,\om)=1, \ p_\si(1,\om)=0,\quad \sigma\in\mathcal{C}
 \eeq
The functions $\cE(k,\om)$ and $ P(k,\om)$ are   holomorphic in a neighborhood of $\overline{\D}$, except, perhaps zero, where  $ P(k,\om)$ may have poles. \\
Moreover, it is well known, see \cite{MS,Te}, that the Jost solutions form a fundamental system of solutions of the finite-difference equation \eqref{eq:30}.  If $|\theta|=1$ and $|\theta|\neq 1$, it follows from \eqref{eq:36} that $e(\theta)$, $\bar{e}(\theta)=e(\theta^{-1})$ are independent solutions of \eqref{eq:30} and any other solution $\{x(k,\theta)\}_{k\geq 0}$ can be expressed as $x(k,\theta)=m(\theta)e(k,\theta)+n(\theta)e(k,\theta^{-1})$, with $m(\theta),n(\theta)$ independent of $k$. Thus, for  $\om \in T_\si^+$ the function $p_\si(k,\om)$ may be expressed in terms of the corresponding Jost solutions:
\beq
\label{eq:53}
p_\si(k,\om)= \frac{\gb_\si(0)}{b_\si}
\frac{e_\si(k,\te_\si(\om))e_\si(1,\te_\si(\om)^{-1}) - e_\si(k,\te_\si(\om)^{-1})e_\si(1,\te_\si(\om))}{\te_\si(\om)^{-1}-\te_\si(\om)}.
\eeq
\\ \\
In order to construct the special solutions we need auxiliary operator-function $T(\om): l^2(\cC) \to l^2(\cC)$ defined as
\beq
\label{eq:54}
T(\om)= \cE(0,\om) - \cR(\la(\om)) B(0) \cE(1, \om),  \ \om \in \D.
\eeq
This function is holomorphic in a vicinity of $\overline{\mathbb{D}}$, except the set $\cO$ of poles of the function $\cR(\lambda(\omega))$. 
\medskip

Denote 
\beqq
\label{eq:55}
\Delta_\si(\om) =\begin{cases} 
                                                                   b_\si , & |\theta_\si(\om)|=1; \\
                                                                  b_\si   (|\te_\si(\om)|^{-2}-1)\sum_{k=1}^\infty|e_\si(k,\te_\si(\om))|^2, &  |\te_\si(\om)|<1,                                                                  
                           \end{cases}
\eeqq
and                            
\beq
\label{eq:56}                                                           
\Delta(\om)=\di\{\Delta_\si(\om)\}_{\si \in \cC}, \quad \Phi(\om)=\di \{\bar{\theta}_\si(\om)-\theta_\si(\om)\}_{\si\in\cC}.
\eeq   
The operator  $\Delta(\om)$ is   positive uniformly with respect $\om \in \overline{ \bD}$, i.e.,
for some $C>0$,
\beqq
\label{eq:57}
\lng \Delta(\om){\mbf x, \mbf x} \rng \geq  C \|{\mbf x }\|^2,
                  \quad {\mbf  x}\in l^2(\cC), \, \om \in \overline{\bD}.
\eeqq
 Besides relation  \eqref{eq:36} now reads
 \beq
\label{eq:58}
B(0)\left \{ \cE(0,\om)\cE(1,\om)^*-\cE(0,\om)^*\cE(1,\om) \right \} =
  \Phi(\om)\Delta(\om),
\eeq                                                      

\begin{Lemma}
\label{L:4.1a}
(See lemma 3.1 in \cite{LM})The following inequality holds for all $\theta\in \bD \setminus \cO$, ${\mbf  x}=(x_\si)_{\si\in \cC}\in l^2(\cC)$
\beq
\label{eq:59}
|\lng \cE(1,\om)^*B(0)T(\om){\mbf x}, {\mbf x}\rng |\geq
     |\Im \lng \cE(1,\om)^*B(0)T(\om){\mbf x}, {\mbf x}\rng |  \geq 
         C\sum_{\si\in \cC} |\bar{\theta}_\si(\om)-\theta_\si(\om)||x_\si|^2.
\eeq
\end{Lemma}
\begin{proof}
We have
\beqq
\label{eq:60}
2\Im   \lng \cE(1,\om)^*B(0)T(\om){\mbf x}, {\mbf x} \rng =
    \lng [ \cE(1,\om)^*B(0)T(\om)- T(\om)^*B(0)\cE(1,\om)]{\mbf x}, {\mbf x} \rng
\eeqq
and, by (\ref{eq:54}),
\begin{multline*}
\label{eq:61}
\cE(1,\omega)^*B(0)T(\om)- T(\om)^*B(0)\cE(1,\om)=
        B(0)(\cE(0,\om)\cE(1, \om)^*-\cE(0, \om)^*\cE(1,\om))+ \\
                 (B(0)\cE(1,\om))^*(\cR(\lambda(\om))^*-\cR(\lambda(\om)))
                         (B(0)\cE(1,\om))
\end{multline*}
The lemma now follows from (\ref{eq:56}) - (\ref{eq:58}) and from the fact that
 \begin{equation*}
 \label{eq:62} 
  \cR(\lambda(\om))^*-\cR(\lambda(\om))=
              b (\bar{\om}- \om)\underbrace {(|\om|^{-2}-1)               \left (
        \sum_{l=1}^M \frac{p_l(\sigma(0))p_l(\nu(0))}{|\lambda_l - \lambda(\om)|^2}
                 \right )_{\sigma, \nu \in \cC}}_{\Delta_1(\om)},
\end{equation*}
here  $\Delta_1(\theta)$ is a non-negative operator.
\end{proof}

 \medskip
\begin{Cor}
Operators $T(\om)$ are invertible for all non-real $\om $ in the open disk ${\bD}\setminus \{\cO\cup \{\pm1\}\}$.
\end{Cor}
\begin{proof}
Indeed, fix an $\omega\in \D\setminus [-1,1]$ and denote $\delta(\omega)=\inf_\sigma \{|\Im \theta_\sigma(\omega)| \}>0$.
Relation \eqref{eq:59} now reads 
\beqq
\label{eq:62a}
\|T(\omega)\mbf x\|>C\delta(\omega)\|\mbf x\|,
\eeqq
which yields  invertibility of $T(\omega)$. 
\end{proof}
Since $T(\omega)$ is an analytic function, we can now claim that $T(\omega)$ is invertible in a vicinity of $\overline{\mathbb{D}}$, except perhaps at a finite set of points.\\ \\ 
Consider the analytic matrix functions :
\beqq
\label{eq:63}
D(\om):= \di\{\te_\si(\om)^{-1}-\te_\si(\om)\}
\eeqq
and 
\beq
\label{eq:64}
U(k,\om)=\left [ -P(k,\om)+\cE(k,\om) T(\om)^{-1} \right ]B B(0)^{-1}D(\om)\cE (1,\om)^{-1}.
\eeq
\begin{Lemma}
Consider the vector function on $\cA_0\cup \Gamma$
 \beq
\label{eq:65}
\psi^\si(k,\omega)= (\psi^\si(\gamma(k),\om))_{\gamma\in \cC}:= U(k,\om)  {\mbf  n}_\si.
\eeq
This function satisfies the boundary condition \eqref{eq:08b} and, for $\om\in T^+_\si$, admits representation  \eqref{eq:23}
along the channels. Thus by \eqref{eq:07}  it may be prolongated to a special solution of \eqref{eq:01}.
\end{Lemma}
\label{L:5}
\begin{proof} Representation  \eqref{eq:23} for $\om\in T^+_\si$ is just a consequence of  \eqref{eq:64}  and \eqref{eq:53}. \\ \\
In order to prove that $\psi^\si(k,\omega)$ meets the boundary condition for $\om\in T^+_\si$, we prove that this condition is met by the whole matrix-function  
$U(k,\om)$:
\beq
\label{eq:66}
U(0,\om)=\cR(\la(\om))B(0)U(1,\om),\quad \ \om \in \D:
\eeq
This implies that   $\psi^\si(k,\omega)$ also  meets the boundary condition.  Relation \eqref{eq:66} is straightforward: after factoring out the inessential factor $B B(0)^{-1}D(\om)\cE (1,\om)^{-1}$ in the definition 
\eqref{eq:64}  it becomes
$$
-I+\cE(0,\om)T(\om)^{-1}=\cR(\la(\om))B(0)\cE(1,\om)T(\om)^{-1},
$$\
which is just the definition of $T(\om)$.\\ \\
So far the special solution is not defined at a point $\omega_0\in T^+_\sigma$ in case $\omega_0$ is a pole of $U(k,\omega)$. This case will be considered in the following sections: in section \ref{SingU(k)} we prove that all poles of $U(k,\omega)$ are simple, later in section \ref{sec8} we show that. if $\omega_0\in T_\sigma^+$, the function $U(k,\omega){\mbf  n}_\si$ is continuos at $\omega_0$ even if $\omega_0$ is a pole of $U(k,\omega)$. In particular
\begin{equation}\label{ast}
\text{Res}_{\omega_0}u_{\nu\sigma}(\omega)=0,\quad \nu\in\mathcal{C}
\end{equation}
\end{proof}
\section{Singularities  of  $U(k,\om)$.}\label{SingU(k)}
  
The matrix function $U(k,\om)$  is  analytic in a vicinity of $\bar \D$,  in particular it has a finite number of poles which belong to $\overline{\D}$.   
 \begin{Lemma}
  \label{L:07}
   There is a finite set $\Omega\subset \overline{\D}$ such that all poles of the matrix function $U(k,\om)$  $\in \overline{ \D}$
 belong to 
$\Omega \cup\{0\}$. In the origin $U(k,\om)$ has pole of order $k$, all poles in $\Omega$ are simple.
 \end{Lemma}  
   
   \begin{proof} The proof follows the pattern of Lemma 4.1 in \cite{LM}. 
  We rewrite \eqref{eq:64} as 
   \begin{multline}
\label{Eq:23a01}
U(k,\om)=  -P(k,\om)B B(0)^{-1}D(\om)\cE (1,\om)^{-1} + \\
\cE(k,\om) T(\om)^{-1}  B B(0)^{-1}D(\om)\cE (1,\om)^{-1}.
\end{multline}
Therefore all poles of $U(k,\om)$ in $\overline{ \D}$ belong to the finite set $\Omega\cup\{0\}$ where $\Omega$ includes all poles of $T(\om)^{-1}$ and $\cE(1,\om)^{-1}$ in $\overline{ \D}$.\\ 
That at the origin the functions $U(0,\om)$ and $U(1,\om)$ have poles of order 0 and 1 respectively follows
from the initial conditions for $P$.
   For $k\geq 2$ the main contribution to singularity of $U(k,\om)$ at zero comes from the first term in the
right-hand side of  \eqref{Eq:23a01} because $P(k,\om)$ is a polynomial of degree $k-2$ with respect to 
   $\la=a-b(\om+\om^{-1})$.\\ 
It now suffices to study the singularities  of the second term in  the right-hand side in \eqref{Eq:23a01}. Let $\Omega$ be the set of all such singularities in $\overline{\mathbb{D}}$. 
These singularities comes from the poles of $T(\omega)^{-1}$, that is, when det$(T(\omega))=0$ and also  from the zeros of det$(\cE(1,\om))=\prod_{\sigma\in\mathcal{C}}e_\sigma(1,\theta_\sigma(\omega))$.\\
Actually  $U(k,\om)$ cannot have poles outside $\T\cup (-1,1)$. This is due to the poles of $T(\omega)^{-1}$ are located in $\T\cup (-1,1)$ and  $e_\sigma(1,\theta_\sigma(\omega))=0$ only if $\theta_\sigma(\omega)\in (-1,1)$, see \eqref{eq:36}, and thus $\omega\in\T\cup(-1,1)$.\\
Let  $\hat\om\in (-1,0)\cup(0,1)$.  
We are going to use \eqref{eq:59} as $\omega$ approaches $\hat\om$. We have 
 $|\te_\si(\om)-\te_\si(\bar\om)|\asymp |\om-\bar\om|$, so for any $\mbf y\in l^2(\cC)$ relation   \eqref{eq:59} with
 \[
 \mbf x= T(\om)^{-1}BB(0)^{-1}D(\om)\cE(1,\om)^{-1} \mbf y
 \]  
 gives 
 \begin{multline*}
| \lng \cE(1,\om)^*\cE(1,\om)^{-1}BD(\om)\mbf y, T(\om)^{-1}BB(0)^{-1}D(\om)\cE(1,\om)^{-1} \mbf y \rng |\geq \\
C |\om-\bar\om| \| T(\om)^{-1}BB(0)^{-1}D(\om)\cE(1,\om)^{-1} \mbf y    \|^2.
 \end{multline*}
Since $\cE(1,\om)^*\cE(1,\om)^{-1}$  is unitary and also $| \mbox{det}D(\om)|$ stays bounded from below near 
 $\hat\om$,  the Schwartz inequality gives
 \[
 \|T(\om)^{-1}BB(0)^{-1}D(\om)\cE(1,\om)^{-1} \mbf y  \| \leq \mbox {Const} \frac{\|\mbf y\|}{|\om-\bar \om|} \
 \mbox{as} \ \om\to \hat\om,
  \]
 this is possible only in case $\hat\om$ is a simple pole of $T(\om)^{-1}BB(0)^{-1}D(\om)\cE(1,\om)^{-1}   $.\\
In  case when $\hat\om\in \T$, the reasoning goes in a similar way, it suffices to let $\om$ approach $\hat\om$ in a way that  $|\te_\si(\om)-\hat\te_\si(\om)| \asymp |\om-\hat\om|$. 
\end{proof}
%
%
%
%
\section{Relation for the scattering coefficients}\label{sec5}

The special solution $\psi^\si(\alpha,\om)$ is now well-defined for all $\om \in \T\setminus\Omega$. Thus
 the non-diagonal scattering coefficients $s_{\si\gamma}(\om)$
$\si\neq \gamma$ are also well-defied for all $\om\in \T\setminus\Omega$. Together with $\psi^\si(\om)$ they are analytic in a vicinity of $\overline{\D}$.\\ \\ 
The scattering coefficients $s_{\si,\si}(\om)$ are so far well-defined for $\om\in T_\si^+$ only. 
For $\om \in T_\si^-$ the corresponding value $\te_\si(\om)$ belongs to the unit disk. One
 can construct (similarly to how this is done in Theorem 1.4.1 in \cite{AM}  for the  continuous case) 
a real-valued
  solution  $e_\si^{(1)}(k,\te_\si)$ of 
\eqref{eq:10} such that 
\beq
\label{eq:66a}
e_\si^{(1)}(k,\te_\si)=\te_\si^{-k} (1+o(1)), \ k\to \infty.
\eeq
This choice is not unique since by adding any multiple of $e_\si$ we obtain a solution which still meets this relation. However we fix some choice of functions 
$e_\si^{(1)}(k,\te_\si)$. The functions $p_\si$ can be represented as 
\beq
\label{eq:66b}
p_\si(k,\om)= \frac{\gb_\si(0)}{b_\si}
\frac{e_\si(k,\te_\si(\om))e_\si^{(1)}(1,\te_\si(\om)) - e_\si^{(1)}(k,\te_\si(\om))e_\si(1,\te_\si(\om))}{\te_\si(\om)^{-1}-\te_\si(\om)},
\eeq
which yields 
 \beqq
 \psi^{\si}(\si(k),\om)=\te_\si(\om)^{-k}(1+o(1)), \ k\to \infty
 \eeqq
as it should be for a special solution.

\begin{Lemma} 
\label{L:6}
For each $\si,\gamma\in \cC$, $\si\neq \gamma$ we have 
\beq
\label{eq:67}
b_\gamma(\te_\gamma(\om)^{-1}-\te_\gamma(\om))s_{\gamma\si}(\om)= b_\si(\te_\si(\om)^{-1}-\te_\si(\om)) s_{\si\gamma}(\om), \ \om\in \overline{\D}\setminus \Omega
\eeq
In addition the scattering coefficients $s_{\gamma,\si}(\om)$ and $s_{\si,\gamma}(\om)$ are continuous up to $\T\setminus \Omega$ \end{Lemma}
\begin{proof}
Relation \eqref{eq:67} follows from  \eqref{eq:38} for $\xi=\psi^\si(\om)$, $\eta=\psi^\gamma(\om)$
if one takes into account  that according  \eqref{eq:23} and \eqref{eq:23a} we have 
 \beqq
\gb_\nu(0)\{\psi^\si,\psi^\gamma\}_{\nu}(0)= \begin{cases}
                        -b_\si (\te_\si(\om)^{-1}-\te_\si(\om))s_{\sigma\gamma}(\omega), & \nu=\si; \\
                         b_\gamma (\te_\gamma(\om)^{-1}-\te_\gamma(\om))s_{\gamma\sigma}(\omega), & \nu=\gamma; \\
                         0,  & \nu\neq\si,\gamma.
                                                                      \end{cases}
\eeqq
Continuity of the scattering coefficients $s_{\gamma\sigma}(\omega)$ follows from the definition of $U(k,\omega)$ and the fact that it has finitely many poles contained in $\Omega$.

\end{proof}

\begin{Cor}\label{Rem:51}
Let $\hat{\omega}\in\overline{\mathbb{D}}$, then 
\begin{equation*}
\Res_{\hat\om} u_{\si\nu}(k,\omega)=\Res_{\hat\om} u_{\nu\si}(k,\omega)=0,\quad\hat{\omega}\in T^{+}_\sigma\cup T_\nu^+,\quad\nu\in\mathcal{C}
\end{equation*}
\end{Cor}
\begin{proof}
Denote
\begin{equation*}
A_{\nu\sigma}=\text{det}(\tilde{T}_{\nu\sigma}(\omega)),\quad \omega\in\overline{\mathbb{D}},\quad \nu,\sigma\in\mathcal{C},
\end{equation*}
where $\tilde{T}_{\nu\sigma}(\omega)$ denotes the matrix which comes from $T(\omega)$ once we have removed the $\sigma-$column and the $\nu-$row. \\
Let now $\om_0\in \T$, $\lambda(\omega_0)$, be a regular point of $\cR(\lambda(\omega))$ and $\det(T(\omega_0))=0$. Then there exists $\mbfx=\{x_\sigma\} \in l^2(\cC)$ such that $T(\om_0)\mbfx=0$, thus the vector function  $\vec{\xi}(k)= \cE(k,\om_0)\mbfx$ satisfies the equations    \eqref{eq:02}  on the channels as well as the boundary  conditions \eqref{eq:08b}. Hence it can be prolongated to a solution of the whole problem \eqref{eq:01}. \\ 
For $\mbfx\in \l^2(\mathcal{C})$, denote $\supp \mbfx =\{\sigma\in\mathcal{C};\text{ } x_\sigma\neq 0 \}$ and  let  $\xi=\{\xi(\alpha)\}_{\alpha\in \cA}$ be a solution to  \eqref{eq:01} with $\la=\la(\om_0)$ which is obtained by prolongation of  
$\vec{\xi}(k)$. Then $\eta(\alpha)= \bar{\xi}(\alpha)$ also solves this problem. Relation \eqref{eq:38}  yields
\[
0= \sum_{\si\in \supp \mbfx} \gb_{\sigma}(0)\{\xi,\eta\}_{\si}(0)= 
                     \sum_{\si\in \supp \mbfx} b_\sigma  |x_\sigma|^2 (\bar{\theta}_\sigma(\om_0)- {\theta}_\sigma(\om_0))
\]
Therefore $\om_0\in T_\si^-$ for each $\sigma\in \supp \mbfx$. Thus $A_{\sigma\nu}(\omega_0)=0$, for $\omega_0\in T_\nu^+$, $\sigma\in\mathcal{C}$.\\
Since the poles of $U(k,\omega)$ are simple and \eqref{eq:36} implies that $e_\nu(1,\theta_\nu(\hat{\omega}))\neq 0$ if $\omega\in T_\nu^+$,  we obtain  
$$
\Res_{\omega_0} u_{\nu\si}(k,\omega)=0 
$$
A simple application of lemma  \ref{L:6} gives us  $\Res_{\omega_0} u_{\sigma\nu}(k,\omega)=0 $
\end{proof}

\begin{Cor}\label{Cor:51}
Let $\omega\in T_\sigma^+\cup T_\nu^+$, $\sigma,\nu\in\mathcal{C}$ and $\sigma\neq\nu$. Then 
\begin{equation*}
s_{\nu\sigma}(\omega^{-1})=\overline{s_{\nu\sigma}}(\omega)
\end{equation*}
In particular $|s_{\nu\sigma(\omega)}|^2=s_{\nu\sigma}(\omega)s_{\nu\sigma}(\omega^{-1})$
\end{Cor}
\begin{proof}
We have $u_{\nu\sigma}(k,\omega)=e_\nu(k,\theta_\nu(\omega))s_{\nu\sigma}(\omega)$ and by construction of the matrix $U(k,\omega)$, we know that 
\begin{equation*}
u_{\nu\sigma}(k,\omega^{-1})=\overline{u_{\nu\sigma}}(k,\omega),\quad \omega\in T_\sigma^+\cup T_\nu^+
\end{equation*}
and the corollary follows.
\end{proof}

\section{Discrete spectra of the operator $\cL$ }\label{sec8}

  \begin{Lemma} 
  \label{L:08}
   Let $\hat\om\in \bar \D\setminus \{0\}$ be a pole  of $U(k,\om)$. Then $\la(\hat\om)=a-b(\hat\om+\hat\om^{-1})$
  is an eigenvalue of \eqref{eq:01}. If in addition $\hat\om\in T_\si^+$, then all elements $u_{\si\nu}(k,\om)$, $\nu\in \cC$ are
  regular at $\hat\om$. 
   \end{Lemma}
   \begin{proof}     
   Let $\hat\om\in \bar\D\setminus \{0\}$ be a (simple)  pole of $U(k,\om)$.  Denote
   \beqq
   \label{Eq:24}
   a_\nu(\om)= \frac {\te_\nu(\hat\om)-\te_\nu(\hat\om)^{-1}}{\te_\nu( \om)-\te_\nu( \om)^{-1}} (\om-\hat\om), \ \nu\in \cC; \quad A(\om)= \di\{a_\nu(\om)\}_{\nu\in \cC}.
   \eeqq
   We then have
   \beqq
   \label{Eq:25}
   \Res_{\hat\om} U(k,\om)=\lim_{\om\to \hat\om}   U(k,\om)A(\om).
   \eeqq
For each $\nu\in \cC$ and for each $\om\neq\hat\om$ the $\nu$-th column of $ U(k,\om)A(\om)$
   $$
 \vec{\phi}^\nu(k,\om)=  \left ( \phi^\nu(\si(k),\om)   \right )_{\si\in\cC}=U(k,\om)A(\om) \mathbf{n}_\nu
   $$
   can be prolongated into $\cA_1$ to a   solution of \eqref{eq:01}  with $\la=\la(\om)$ according \eqref{eq:07}:
   \beqq
   \label{Eq:26}
   \phi^\nu(\alpha, \om)=  \sum_{\gamma\in \cC} r(\alpha, \gamma(0); \la) \gb_\gamma(0)\phi^\nu(\gamma(1),\om); \ \alpha \in \cA_1   
   \eeqq
   Since $U(k,\omega)$ has a simple pole at $\hat{\omega}$, there exists the limit
\begin{equation*}
 \vec{\phi}^\nu(k, \hat{\om})=\lim_{\omega\to\hat{\omega}} \vec{\phi}^\nu(k, \om)
\end{equation*}
   \medskip
   
   {\bf Claim} {\em The vector $ \vec{\phi}^\nu(k,\hat\om)$   also can be prolongated to a  solution of the problem 
   \eqref{eq:01}}.\\ 
In case $\lambda(\hat{\omega})$ is not a pole of $\cR(\lambda(\omega))$, the prolongation is straightfoward. By \eqref{eq:07}, if $\lambda(\hat{\omega})$ is at the same time an eigenvalue of $\cL_1$ one can apply the same reasonings as in lemma 4.3 in \cite{LM}. We omit the details.\\ \\
Let now 
\begin{equation*}
T(\omega)^{-1}=(\tau_{\sigma\nu}(\omega))_{\sigma,\nu\in\mathcal{C}},
\end{equation*}
and denote
  \beq
  \label{EQ:01}
   h_{\si\nu}(\om)= -\frac{b_\nu}{\gb_\nu(0)}\frac{\te_\nu(\hat\om)-\te_\nu(\hat\om)^{-1}}{e_\nu(1, \te_\nu(\om))}(\om-\hat \om)\tau_{\si\nu}(\om).
   \eeq
It follows from  \eqref{eq:64}  that for $\si\neq \nu$ we have
  \beq
 \label{Eq:33}
  \Res_{\hat\om} u_{\si\nu} (k,\cdot)= \phi^\nu(\si(k),\hat\om) =  e_\si(k,\hat\om)h_{\sigma\nu}(\hat{\omega})
  \eeq 
here we denote $\gm_{\si\nu}(\hat\om)=h_{\si\nu}(\hat\om)$.\\ \\
If $e_\nu(1,\te_\nu(\hat\om))\neq 0$, representation \eqref{Eq:33} is valid for $\si=\nu$ as well.
Assume that  $e_\nu(1,\te_\nu(\hat\om))= 0$.   Then
\beq
\label{Eq:33ab}
 p_\nu(k,\hat\om)e_\nu(0,\te_\nu(\hat\om))=e_\nu(k, \te_\nu(\hat\om)),
\eeq 
because the expressions in both sides satisfy the same recurrence equation  and the same initial conditions, and also it follows from lemma \ref{le:02a}
that $\dot e_\nu(1,\te_\nu(\hat\om))\neq 0$ because    $\ e_\nu(1,\te_\sigma(\hat{\omega}) )$ and $\dot e_\nu(1,\te_\sigma(\hat{\omega}) )$  cannot vanish simultaneously. 
Thus we again obtain \eqref{Eq:33}, yet now
  \beqq
  \label{Eq:37}
  m_{\nu\nu}(\hat\om)= \frac{b_\nu}{\gb_\nu(0)}\frac{\te_\nu(\hat\om)-\te_\nu(\hat\om)^{-1}}
  {\{e_\nu(\hat\om),\dot e_\nu(\hat\om)\}_\nu(0)}+h_{\nu\nu}(\hat\om).
  \eeqq
  Representation  \eqref{Eq:33}  is now valid for all $\nu,\si\in \cC$. \\ \\  
If $\hat\om \in \D$ it follows from  \eqref{Eq:33} that $\phi^\nu(\al,\hat\om)$ is an eigenfunction of $\cL$ with $\la(\hat\om)$ as eigenvalue. If $\hat\om\in \T$, we may have $\te_\si(\hat\om)\in \T$, i.e., $\hat{\omega}\in T^+_\sigma$ for some $\si\in \cC$.  It follows from corollary \ref{Rem:51} that    for such $\si$ we have $m_{\si\nu}(\hat\om)=0$, in particular
  \beqq
  \label{Eq:37b01}   
  \Res_{\hat\om} u_{\si\nu}(k, \cdot)=\phi^\nu(\si(k), \hat\om)=0, \ \hat\om \in \T,   \te_\si(\hat\om) \in \T.
  \eeqq
   and, again $\phi^\nu(\al,\hat\om)$ is an eigenfunction of $\cL$ with $\la(\hat\om)$ as eigenvalue.
   \end{proof}

Consider the matrix 
 \beqq
 \label{Eq:37b02}
 \gm(\hat\om)= \left ( m_{\si\nu}(\hat\om) \right )_{\si,\nu\in \cC}
  \eeqq
Properties of  $\gm(\hat\om)$ are summarized in the statement below

 \begin{Lemma}
 \label{L:09}
 Let $\hat\om\in \Omega$.  Then 
 \beq
 \label{Eq:39}  
 \Res U(k,\hat\om)= \cE(k, \hat\om)\gm(\hat\om), \ k=0,1,2, \ldots.
 \eeq
   The diagonal  elements $m(\nu,\nu; \hat\om)$ satisfy 
  \beq
  \label{Eq:36}
  \|\phi^\nu(\hat\om)\|^2=-\frac{b_\nu(1-\te_\nu(\hat\om)^{-2})}{b(1- \bar{\hat{\omega}}^{-2})}\theta_\nu(\hat\om)
        m_{\nu\nu}(\hat\om),
 \eeq 
   where $\phi^\nu=\phi^\nu(\al, \hat\om)$ is the eigenvector  of $\cL$, corresponding to the eigenvalue $\la(\hat\om)$ and such that 
  \beq
  \label{Eqq:40}
  \phi^\nu(\si(k), \hat\om)   = e_\si(k,\te_\si(\hat\om)) m(\si,\nu;\hat\om) \ \si\in \cC, \ k\geq 0.
 \eeq 
   \end{Lemma}
 
  \begin{proof}
   Relations  \eqref{Eq:39} and \eqref{Eqq:40}  are already established in lemma \ref{L:08}. It remains to prove \eqref{Eq:36}.\\ \\ 
We     apply Lemma \ref{le:02}  with $\xi(\al)= \phi^\nu(\al,\hat\om)  $:
 \beq
 \label{Eq:32}   
    \dot\la(\hat\om)\sum_{\al\in\cA_1} |\phi^\nu(\al,\hat\om)|^2=
            \sum_{\si\in \cC}\gb_\si(0)\{\dot\phi^\nu(\hat\om), \overline{\phi^\nu(\hat\om)} \}_\si(0)
\eeq             
and calculate the Wronskians  in the right-hand side of this equality.\\ \\ 
We then have
 \beq
 \label{Eq:28}
  \phi^\nu(\si(k),\om) = p_\si(k,\om)h_\nu(\om)\delta_{\si,\nu}+e_\si(k,\om)h_{\si\nu}(\om),
  \eeq
with
  \beqq
  \label{Eq:29a}
  h_\nu(\om)=\frac{b_\nu}{\gb_\nu(0)}\frac{\te_\nu(\hat\om)-\te_\nu(\hat\om)^{-1}}{e_\nu(1, \te_\nu(\om))}(\om-\hat \om),  
  \eeqq
and $h_{\si\nu}$ is already defined in \eqref{EQ:01}.\\ \\
Let $e_\nu(1,\te_\nu(\hat\om))\neq 0$. Then, for all $\si,\nu\in \cC$,
    \beq
  \label{Eq:30}
  h_\nu(\hat\om)=0, \ \phi^\nu(\si(k),\hat\om)= e_\si(k,\te_\si(\hat\om))h_{\si\nu}(\hat\om),
  \eeq
  and 
    \beq
    \label{Eq:30a}
  m_{\si\nu}(\hat\om)=h_{\si\nu}(\hat\om), \  \gm(\hat\om)=\left ( m_{\si\nu}(\hat\om)\right )_{\si,\nu\in\cC} 
   \eeq    
We use \eqref{Eq:28}, \eqref{Eq:30}, \eqref{Eq:30a},  \eqref{Eq:39}  and  that $\dot p_k(\om)=0$ for $k=0,1$
as follows from \eqref{eq:52}:

\beq
 \label{Eq:34}
 \dot \phi^\nu(\si(k),\om) = p_\si(k,\om)\dot h_\nu(\om)\delta_{\si,\nu}+\dot e_\si(k,\om)h_{\si\nu}(\om)+
 e_\si(k,\om)\dot h_{\si\nu}(\om), \ k=0,1.
   \eeq     
 Besides 
 \beqq
 \label{Eq:35}
 \dot h_\nu(\hat \om)=\frac{b_\nu}{\gb_\nu(0)}\frac{\te_\nu(\hat\om)-\te_\nu(\hat \om)^{-1}}{e_\nu(1,\te_\nu(\hat\om))}.
 \eeqq 
 Since $e_\nu(1,\te_\nu(\hat\om))\neq 0$, we also have  $h_\nu(\hat\om)=0$ and, according to \eqref{eq:38a07}  and \eqref{Eq:33},

 \begin{multline}
    \label{Eq:35a}
\gb_\si(0)\{\dot\phi^\nu(\hat\om), \overline{\phi^\nu(\hat\om)} \}_\si(0)= 
b_\nu (\te_\nu(\hat\om)-\te_\nu(\hat \om)^{-1})\overline{m_{\si\nu}(\hat\om)} \delta_{\si,\nu}  + \\
         \gb_\si(0)  \{\dot e_\si (\hat\om), e_\si(\te_\si(\hat\om))\}_\si(0) |m_{\si\nu}(\hat\om)|^2 = \\
 b_\nu (\te_\nu(\hat\om)-\te_\nu(\hat \om)^{-1})\overline{m_{\si\nu}(\hat\om)} \delta_{\si,\nu} - 
             \dot\la(\hat\om)\sum_{k=1}^\infty   |\phi^\nu(\si(k), \te_\si(\hat\om))|^2.     
    \end{multline}
  We can now return to   \eqref{Eq:32}  in order to obtain normalization condition \eqref{Eq:36} for the matrix $\gm$:

   \medskip
   
 In the case $e_\nu(1,\hat\om)=0$ representation  \eqref{Eq:28} is still valid, yet 
 \[
 h_\nu(\hat\om)=\frac{b_\nu}{\gb_\nu(0)}\frac{\te_\nu(\hat\om)-\te_\nu(\hat\om)^{-1}}{\{e_\nu,\dot e_\nu\}_\nu(0)}e_\nu(0,\hat\om)\neq 0.
 \]
  Taking \eqref{Eq:33ab} into account   
   we again obtain \eqref{Eq:33}, yet now
  \beqq
  \label{Eq:37}
  m_{\si\nu}(\hat\om)= \frac{b_\nu}{\gb_\nu(0)}\frac{\te_\nu(\hat\om)-\te_\nu(\hat\om)^{-1}}
  {\{e_\nu(\hat\om),\dot e_\nu(\hat\om)\}_\nu(0)}\delta_{\si,\nu}+h_{\si\nu}(\hat\om).
  \eeqq
  Relation  \eqref{Eq:34}  is still valid and for $\nu\neq \si$ we arrive to relation  \eqref{Eq:35a}.\\ \\
For $\si=\nu$ we have
  \begin{multline*}
  \label{Eq:38}
  \gb_\nu(0)\{\dot\phi^\nu(\hat\om), \overline{\phi^\nu(\hat\om)} \}_\nu(0)=
                 \gb_\nu(0)\{\dot e_\nu(\hat\om)h_2(\hat\om),\overline{e_\nu(\hat\om)}\}_\nu(0) \overline{m_{\nu\nu}(\hat{\om})}= \\
         \gb_\nu(0)\{\dot e_\nu(\hat\om),\overline{e_\nu(\hat\om)}\}_\nu(0) |m(\nu,\nu;\om)|^2-
                \ {b_\nu} ({\te_\nu(\hat\om)-\te_\nu(\hat\om)^{-1}})\overline{m_{\nu\nu}(\hat{\om})},                 
        \end{multline*}         
  and we again arrive to \eqref{Eq:35a}.\\ \\  
Now one can complete the proof in the same way as if $e_\nu(1,\te_\nu(\hat\om))\neq 0$.
  
  \end{proof}

\begin{Rem}
The matrix $\overline{\gm(\hat\om)}$ corresponds to the point of discrete spectra $\lambda(\hat{\omega})$. Its columns are normalized eigenfunctions. This normalizations is defined by relation \ref{Eq:36} and is therefore unique.\\
We will see in the next section that  $m_{\nu\nu}(\hat{\omega})$, i.e. the energies of the normalized eigenfunctions, are the quantities which participate in the equations for the inverse scattering problem
\end{Rem}
 %
%
%
%
%
%
%
%
\section{Equations of the inverse scattering problem}\label{sec9}
\subsection{}

 In order to obtain equations of the inverse scattering problem we introduce the matrix function  
\beq
\label{Eq:01}
\Delta_l(\om):= \di\left \{\te_\nu(\om)^{l-1}\frac{d\te_\nu(\om)}{d\om}\right \}_{\nu\in \cC},   \ l\in \Z
\eeq
and consider the integral 
\beqq
\label{Eq:02}
\cJ(l,k) = \left ( j_{\nu\si}(l,k) \right )_{\nu,\si\in \cC}= \frac 1 {2\pi\I} \int_\T \Delta_l(\om) U(k,\om) d\om.
\eeqq

Since $\T$ may contain (simple) poles of $ U(k,\om) $ this integral as well as all integrals in this section is considered in principal value. We will calculate this integral in two ways: through the residues, this would correspond to the contribution of the discrete spectra, and through the scattering coefficients, this would correspond to the contribution of the continuos spectra.\\ \\ 
Comparing two different expressions for $\cJ_{l,k}$ leads one to the equations of the inverse scattering problem.

\medskip

Let as before   $U(k,\om)= (u_{\nu\si}(k,\om))_{\nu,\si}$. Then
\begin{multline}
\label{Eq:02a01}
 j_{\nu\si}(l,k)=  \frac 1 {2\pi\I}\int_\T \theta_\nu(\om)^{l-1}u_{\nu\si}(k,\om)\frac{d\te_\nu(\om)}{d\om}d\om= \\
           \frac 1 {2\pi\I}\int_\T \theta_\nu(\om)^{l-1}\psi^\si_\nu(k,\om)\frac{d\te_\nu(\om)}{d\om}d\om,
\end{multline}
here $\psi^\nu(\cdot, \om)$ is the special solution, defined in \eqref{eq:65}.

\medskip

Assume first that $\si\neq \nu$. Then, according to \eqref{eq:23} and \eqref{eq:23a} 
\beqq
\label{Eq:02a02}
j_{\nu\si}(l,k) =\frac 1 {2\pi\I}\int_\T \theta_\nu(\om)^{l-1}e_\nu(k,\te_\nu(\om))s_{\si\nu}(\om)\frac{d\te_\nu(\om)}{d\om}d\om,
\eeqq
Consider the function 
\beq
\label{Eq:02a03}
q_{\nu\si}(n)=\frac 1 {2\pi\I}\int_\T s_{\si\nu}(\omega) \theta_\nu(\om)^{n-1}\frac {d\te_\nu(\om)}{d\om}d\om.
\eeq
Relation \eqref{eq:13}  now yields
\beq
\label{Eq:02a04}
 j_{\nu\si}(l,k)=c_\nu(k)\sum_{m\geq k} q_{\nu\si}(m+l)a_\nu(k,m),
\eeq
this is the desired expression.

\begin{Rem}
The functions $q_{\nu,\si}$ cannot be determined from the spectral data generally 
speaking. However relation 
\eqref{Eq:02a04} determines the \textbf{structure} of the equation of the inverse scattering problem. Later in section \ref{sec:62} we will use this structure to get rid of the functions which cannot be observed from the spectral data. 
\end{Rem}
Let now $\nu=\si$ and   $T(\om)^{-1}= (\tau_{\si\nu}(\om))_{\si,\nu}$.  Relations  \eqref{eq:64} \eqref{Eq:01}  yield
 \begin{multline}
\label{Eq:05}
j_{\si\si}(k,l)=   \\ \frac1 {2\pi\I} \int_\T\Bigl \{ \bigl [
       -p_\si(k,\om) + e_\si(k,\te_\si)\tau_{\si\si}(\om)
                                                       \bigr ] \frac{b_\si(\theta_\sigma^{-1}-\theta_\sigma)}{\gb_\si(0)e_\si(1,\te_\si)} \te_\si(\om)^{l-1}\frac{d\te_\si(\om)}{d\om} \Bigr \} d\om = \\
     \frac1 {2\pi\I}    \Bigl ( \int_{T_\si^+} +  \int_{T_\si^-}\Bigr) \Bigl \{   \  \cdot \ \Bigr \}   =    j_{\si\si}^+(k,l) + j_{\si\si}^-(k,l).                                        
\end{multline}    
As $\om$ runs over  $T_\si^+$ the function $\te_\si(\om)$ runs over the whole $\T$. Let $\om(\te_\si)$ be the inverse function. 
Relations \eqref{Eq:02a01} together with  \eqref{eq:23} yield
\beqq
\label{Eq:06}
 j_{\si\si}^+(k,l) = \frac1 {2\pi\I} \int_\T \bigl [   e_\sigma(k, \theta_\sigma^{-1}) + s_{\sigma \sigma}(\om(\te_\si))e_\sigma(k, \theta_\sigma)
   \bigr ] \te_\si^{l-1}  d\te_\si.
\eeqq
Besides it follows from corollary \ref{Rem:51}  that $s_{\si\si}(\om)$ is bounded on $\T_\si^+$.\\ 
Consider the Fourier series of $ s_{\sigma \sigma}(\om(\te_\si))$:
\beqq
\label{Eq:07}
 s_{\sigma \sigma}(\om(\te_\si))=\sum_{n=-\infty}^\infty \tilde{s}_\si(n) \te^{-n}; \ 
          \tilde{s}_\si(n) = \frac 1 {2 \pi\I}  \int_\T  s_{\sigma \sigma}(\om(\te_\si)) \te_\si^{n-1}  d \te_\si.
\eeqq
Together with \eqref{eq:13}  this yields
\begin{multline*}
\label{Eq:08}
                                       e_\sigma(k, \theta_\sigma^{-1}) + s_{\sigma \sigma}(\om(\te_\si))e_\sigma(k, \theta_\sigma)=   \\
  c_\si(k) \sum_m \bigl \{ a_\si(k,m) + \sum_n a_\si(k,n) \tilde{s}_\si(m+n) \bigr \} \te_\si^{-m},
\end{multline*}
and finally
\beq
\label{Eq:09}
 j_{\si\si}^+(k,l) = c_\si(k) \left [
                   a_\si(k,l)+\sum_{n=-\infty}^\infty  a_\si(k,n) \tilde{s}_\si(l+n)
                                         \right ]  .
 \eeq                                        
                                                                             
 \medskip

 We now study $j^-_{\si\si}(k,l)$.  Denote
 \beqq
 \label{Eq:10}
 a_1(k,\om)= - p_\si(k,\om) \frac{b_\si}{\gb_\si(0)}e_\si(1,\te_\si)^{-1}(\theta_\sigma^{-1}-\theta_\sigma);
 \eeqq
 \beqq
 \label{Eq:11}
 a_2(\om)=\tau_{\si\si}(\om)\frac{b_\si}{\gb_\si(0)}e_\si(1,\te_\si)^{-1}(\theta_\sigma^{-1}-\theta_\sigma).
 \eeqq 
We then have
 \beqq
 \label{Eq:11a01}
 \psi^\si_\si(k,\om)= a_1(k,\om)+e_\si(k,\te_\si(\om))a_2(\om),
 \eeqq
 and
 \begin{multline*}
 \label{Eq:11a02}
  j_{\si\si}^-(k,l)=  \frac 1 {2 \pi\I} \int_{T_\si^-}  a_1(k,\om) \te_\si(\om)^{l-1}\frac{d\te_\si}{d\om} d\om+ \\
                                \frac 1 {2 \pi\I} \int_{T_\si^-}  a_2( \om) e_\si(k,\te_\si(\om))\te_\si(\om)^{l-1}\frac{d\te_\si}{d\om} d\om.
\end{multline*}                
 As $\om$ runs over $T_\si^-$, the function $\te_\si(\om)$ runs twice  in the opposite directions 
 over  $J_\si=\te_\si(T_\si^-)\subset\R$  and   the values $\te_\si(\om)$ and  
$\te_\si(\om^{-1})$ coinside. Respectively $e_\si(k,\te_\si(\om)), p_\si(k,\om) \in \R$ and 
$e_\si(k,\te_\si(\om))=e_\si(k,\te_\si(\om^{-1}))$, $ p_\si(k,\om) = p_\si(k,\om^{-1}) $, thus
 $a_1(k,\om)=a_1(k,\om^{-1})$.  Besides $ a_2(\om) = \ov{a_2(\om^{-1})}$,    since 
 $\psi^\si(k,\om)= \ov{\psi^\si(k,\om^{-1})}$ for $\omega\in \T$.\\ \\ 
Therefore 
\beq
\label{Eq:13}
  j_{\si\si}^-(k,l)=  \frac 1 {2  \pi\I} \int_{\Gamma_\si}   [ 
  a_2(\om)-a_2(\om^{-1})      ]  e_\si(k,\te_\si)                                                                                 
                                                                        \te_\si^{l-1}\frac{d\te_\si}{d\om} d\om,       
\eeq
where $\Gamma_\si=T_\si^-\cap \C_+$.\\ \\
The function $a_1(k,\om)$ satisfies equation \eqref{eq:02} on the channel $\si$ and 
relations \eqref{eq:66a}, \eqref{eq:66b} yield
\beqq
\label{Eq:13a}
a_1(k,\om)= \te_\si(\om)^{-k}(1+o(1)).
\eeqq
Therefore 

\beqq
\label{Eq:14}
\gb_\si(0) \{\psi^\si(\om),\psi^\si(\om^{-1})\}_\si(0)= b_\si (\te_\si(\om)^{-1}-\te_\si(\om)) 
         [ a_2(\om)-a_2(\om^{-1})].
          \eeqq       
 On the other hand by  \eqref{eq:38}
 \beqq
 \label{Eq:15}
 \gb_\si(0) \{\psi^\si(\om),\psi^\si(\om^{-1})\}_\si(0)=- \sum_{\nu\in \cC\setminus \{\si\}} 
                \gb_\nu(0) \{\psi^\si(\om),\psi^\si(\om^{-1})\}_\nu(0).
 \eeqq
   We have 
   \beqq
   \label{Eq:16}
     \{\psi^\si(\om),\psi^\si(\om^{-1})\}_\nu(0)=0,           \ \om \in T_\nu^-.
     \eeqq
   For $\om\in T_\nu^+\cap T_\sigma^-$ we obtain

  \begin{multline*}
\label{Eq:17}
\gb_\nu(0)\{\psi^\si(\om),\psi^\si(\om^{-1})\}_\nu(0)=  \\
                            \gb_\nu(0) \{s_{\nu\sigma}(\om)e_\nu(\te_\nu(\om)),    s_{\nu\sigma}(\om^{-1})e_\nu(\te_\nu(\om)^{-1})  \}_\nu(0)  \\
                |s_{\nu\sigma}(\om)|^2 \gb_\nu(0)\{ e_\nu(\te_\nu(\om)),    e_\nu(\te_\nu(\om)^{-1})  \}_\nu(0) = \\
 |s_{\nu\sigma}(\om)|^2 b_\nu (\te_\nu(\om)^{-1}-\te_\nu(\om)) =
 - |s_{\si\nu}(\om)|^2 \frac { ( b_\si(\te_\si^{-1}-\te_\si))^2}{b_\nu(\te_\nu^{-1}-\te_\nu)}.    
   \end{multline*}
Now, for  $\te\in J_\si$, we define $\om_\si(\te)=W^{-1}\circ W_\si (\te)$, here the functions $W$, 
   $W_\si$ are defined by \eqref{eq:19}
   and the branch is chosen so that $\om_\si(\te)\in \Gamma_\si$. Let
   \beqq
   N_\si(\te)= \{ \nu; \ \om_\si(\te)\in T_\nu^+ \}, \ \te\in J_\si .
              \eeqq
Denote 
\beqq
\label{Eq:18}
  \Phi_\si(\te)=  \sum_{\nu\in N_\si(\te)}   
         |s_{\si\nu}(\om_\si)|^2  \left (\frac {  b_\si(\te^{-1}-\te) }{b_\nu(\te_\nu(\om_\si)^{-1}-\te_\nu(\om_\si))} \right ), \
        \om_\si(\te)\in T_\si^-. 
\eeqq           
  We finally obtain 
 \beq
   \label{Eq:19}
 a_2(\om)-a_2 (\om^{-1})=  \Phi_\si(\te_\si(\om))  
  \eeq
   this function is expressed via the spectral data.

 \begin{Rem} All functions $|s_{\nu\si}(\om)|$ participating in \eqref{Eq:19} are  well  defined and continuous. In addition,  $s_{\nu\sigma}(\omega^{-1})=\overline{s_{\nu\sigma}}(\omega)$ for $\omega\in\ T_\nu^+$.
\end{Rem}
   Relation \eqref{Eq:13}  now takes the form
   \beqq
   \label{Eq:20}
   j_{\si\si}^-(k,l)= \frac{ 1 }{2\pi\I} \int_{J_\si} e_\si(k,\te) \Phi_\si(\te) \te^{l-1}  d\te,
   \eeqq
   and with account \eqref{eq:13}  we  obtain
   \beq
   \label{Eq:21}
   j_{\si\si}^-(k,l)= c_\si(k)\sum_n a_\si(k,n) q_{\si\si}(n+l),
   \eeq
   where
   \beq
   \label{Eq:22}
   q_{\si\si}(n)= \frac {1 }{2\pi\I} \int_{J_\si} \Phi_\si(\te)\te^{n-1} d\te.
   \eeq
   Combining \eqref{Eq:05}, \eqref{Eq:09},  and     \eqref{Eq:21} we finally obtain
   \beq
   \label{Eq:23}
    j_{\si\si} (k,l) = c_\si(k) \left [
                   a_\si(k,l)+\sum_{n=-\infty}^\infty  a_\si(k,n) \left( \tilde{s}_\si(l+n)+q_{\si\si}(n+l) \right )
                                         \right ],  
 \eeq                                        
the functions $\tilde{s}_\sigma(\cdot)$ and $q_{\sigma,\sigma}(\cdot)$ are defined through the scattering data.
%
%

\subsection{}\label{sec:62}

Let now
\beqq
\label{Eq:40}
\Theta(\om)= \di \{ \te_\si(\om) \}_{\si\in \cC}.
\eeqq
We use the representation  \eqref{eq:13} for $\cE(k,\om)$,
\beqq
\label{Eq:41}
\cE(k,\om)=C(k)\sum_{m=-\infty}^\infty A(k,m)\Theta(\om)^m.
\eeqq
The coefficients $C(k)$ and $A(k,m)$ are diagonal matrices
\beqq
\label{Eq:42}
C(k)=\di\{c_\si(k)\}_{\si\in \cC}, \ A(k,m)=\di\{a_\si(k,m)\}_{\si\in \cC},
\eeqq
and also $A(k,m)=0$ for $m<k$.\\ \\
It follows from \eqref{Eq:02a04}  and \eqref{Eq:23} that
\begin{multline}
\label{Eq:43}
\cJ(l,k)=\frac 1 {2\pi\I}\int_\T \Delta_l(\om)U(k,\om)d\om= \\
     C(k)  \left\{   A(k,l)+\sum_{l=-\infty}^\infty A(k,m) Z(l+m) \right\},
\end{multline}     
  here the integral is taken as a principal value and   the matrix $Z(n)=\left ( z_{\nu,\si}(n) \right )_{\nu,\si\in \cC}$ is given by the relations
  \begin{gather}
  z_{\nu\si}(n)=q_{\nu\si}(n), \ \nu\neq \si,  
  \label{Eq:44} \\
   z_{\si\si}(n)=\tilde s_\si(n)+q_{\si\si}(n), 
    \label{Eq:45}
    \end{gather} 
    here $\tilde s_\si(n)$ are the Fourier coefficients of the reflection coefficient $s_{\si\si}$ with respect to $\theta_\sigma$ and the function $q_{\nu\si}$ is defined in \eqref{Eq:02a03}  and \eqref{Eq:22}.

  \medskip
  
  Let now $\Omega\subset \bar \D$ be the set of all poles of $U(k,\om)$. It follows from \eqref{eq:64} that $\overline{U(k,\omega)}=U(k,\bar{\omega})$, thus  for each $\hat\om\in \Omega\cap \T$ the point $\bar{\hat\om}$ also belongs to $\Omega$. Moreover, for such $\hat\om$ we have $\overline{\text{Res}_{\hat{\omega}}U(k,\omega)}=\text{Res}_{\bar{\hat{\omega}}}U(k,\omega)$. 
Denote $\tilde \Omega = \{\hat\om \in \Omega: |\hat\om| < 1\} \cup \{\hat\omega\in \T\cap \Omega: \Im \hat\omega >0\}$. We use \eqref{Eq:39}  and \eqref{Eq:01}:

\beqq
 \label{Eq:46}
  \cJ(k,l)=\sum_{\hat\om\in \tilde\Omega} \Delta_l(\hat\om)\cE(k,\hat\om)\gm(\hat\om)=
       C(k)\sum_mA(k,m)M(m+l),
   \eeqq
     where
     \beq
     \label{Eq:47}
   M(n)=\sum_{\hat\om\in \tilde\Omega} \di\left \{ \frac {d\te_\si}{d\om}(\hat\om) \right \}_{\si\in \cC}
    \Theta(\hat\om)^{n-1}\Re\gm(\hat\om).
   \eeq   
     By compare this to \eqref{Eq:43} and taking into account that $A(k,m)=0$ for $m<k$ and $A(k,k)=I$, we obtain a system of equations
     \beqq
     \label{Eq:48}
     A(k,m)+\sum_{s=k}^\infty A(k,s)F(s+m)=0,\quad  m=k+1,k+2, \ldots \,
     \eeqq
     where 
     \beq
     \label{Eq:49}
     F(n)=Z(n)-M(n).
     \eeq
     Since $A(k,k)=I$, this relation can be written as 
     \beq
     \label{Eq:50}
     F(k+m)+A(k,m)+\sum_{s=k+1}^\infty A(k,s)F(s+m)=0, \quad k=1,2, \ldots,\quad  m>k. 
     \eeq   
     The matrices  $A(k,m)$ are diagonal. The diagonal elements of $F(n)$, can be expressed through the spectral data. The non-diagonal elements of the matrices $F(n)$ vanish, as follows from the lemma below.
 \begin{Lemma}
 \label{L:7.1}
The matrices $F(n)$ are diagonal for all $n\geq 1$: $F(n)=\di\{f_\si(n)\}_{\si\in \cC}$.   and their diagonal elements
$f_\sigma(n)$ are determined by $a_\sigma(k,m)$, $m>k\geq \Big [\frac{n-1}2\Big ]$ only.
 \end{Lemma}
  This statement is proved in \cite{LM}, (Lemma 5.1) and we omit the proof.

\begin{Th}
\label{Thm:5.1}
The following properties hold for the systems  under consideration

1) Equations \eqref{Eq:50} split into a system of independent scalar equations
\beq
\label{Eq:6.18a}
f_\nu(k+m)+a_\nu(k,m)+\sum_{s=k+1}^\infty a_\nu(k,s)f_\nu(s+m)=0, \,
m\geq k+1\geq 1,
\eeq
here $f_\nu(n)$ is defined in \eqref{Eq:49}, \eqref{Eq:47},  \eqref{Eq:45}.

 \medskip

2) For each $k\geq 0$ equations \eqref{Eq:6.18a} has unique solutions $a_\sigma(k,m)$.
\end{Th}

\begin{proof}
The first statement follows directly from relation \eqref{Eq:49} which defines $F(n)$  and Lemma \ref{L:7.1}.
It follows from the same lemma that the functions $f_\nu(n)$, $n\geq 1$ are uniquely defined by the coefficients $\gb_\nu(k)$, $\ga_\nu(k)$, $k\geq 1$.
Therefore  \eqref{Eq:6.18a} coincide with equations for inverse scattering problem for equation \eqref{eq:02} with boundary condition $\xi(\nu(0))=0$.
It is well-known (see for example \cite{VA}) that the later has unique solution. 
\end{proof}
   

\section{Concluding remarks}\label{sec10}
In case that the continuos spectra $I=\cup_{\sigma\in\mathcal{C}}[a_\sigma-2b_\sigma,a_\sigma+2b_\sigma]$ splits into a number of disjoints intervals, one can repeat the procedure  separately for each connected component of $I$.\\
If, say, $I^{(0)}$ is  a connected component of $I$ and $\sigma$ is a channel corresponding to this component, then each wave incoming along $\sigma$ generates decaying waves in all channels which correspond to other connected components of $I$. We omit the details.\\ \\
So far we have discussed reconstruction of the part $\mathcal{A}_0$ of the whole system, that is, the channels. This information is, generally speaking, insufficient for reconstruction the whole matrix $\cL$. However if the matrix $\cL_1$ corresponding to the "central" part of the system is sufficiently sparse and also we know the matrix $B(0)=\text{diag}\{\gb_\sigma(0)\}_{\sigma\in\mathcal{C}}$ which realizes connections between the channels and the central part of the system, the whole matrix $\cL$ can be recovered from the scattering data. We refer the reader to Chapter 11 in \cite{MS}, where  statemets of such type are obtained.

\end{document}